\documentclass[12pt]{article}

\usepackage{fullpage,amsmath,amssymb,amsthm}
\usepackage{enumitem}
\usepackage{hyperref}

\usepackage{tikz}
\usepackage{caption}
\newcommand{\ep}{\varepsilon}

\newcommand{\A}{\mathcal{A}}

\newcommand{\Y}{\mathcal{Y}}
\newcommand{\Z}{\mathcal{Z}}
\renewcommand{\P}{\mathcal{P}}

\renewcommand{\bar}[1]{\overline{#1}}
\newcommand{\floor}[1]{\left\lfloor#1\right\rfloor}
\newcommand{\ceil}[1]{\left\lceil#1\right\rceil}

\DeclareMathOperator{\cost}{Cost}

\newtheorem{theorem}{Theorem}

\newtheorem{lemma}[theorem]{Lemma}
\newtheorem{definition}[theorem]{Definition}
\newtheorem{conjecture}[theorem]{Conjecture}
\newtheorem{proposition}[theorem]{Proposition}
\newtheorem{claim}[theorem]{Claim}

\title{A sharp threshold for a random version of Sperner's Theorem}
\author{J\'ozsef Balogh\footnote{Department of Mathematics, University of Illinois at Urbana-Champaign, Urbana, Illinois 61801, USA. E-mail: \texttt{jobal@illinois.edu}. Research supported by NSF RTG Grant DMS-1937241, NSF Grant DMS-1764123, Arnold O. Beckman Research Award (UIUC Campus Research Board RB 18132), the Langan Scholar Fund (UIUC), and the Simons Fellowship.} \and Robert A.\ Krueger\footnote{Department of Mathematics, University of Illinois at Urbana-Champaign, Urbana, Illinois 61801, USA. Email: \texttt{rak5@illinois.edu}. Research supported by the National Science Foundation Graduate Research Fellowship Program under Grant No. DGE 21-46756.}}
\date{September 20, 2023}

\begin{document}

\maketitle

\begin{abstract}
The Boolean lattice $\mathcal{P}(n)$ consists of all subsets of $[n] = \{1,\dots, n\}$ partially ordered under the containment relation. Sperner's Theorem states that the largest antichain of the Boolean lattice is given by a middle layer: the collection of all sets of size $\floor{n/2}$, or also, if $n$ is odd, the collection of all sets of size $\ceil{n/2}$. Given $p$, choose each subset of $[n]$ with probability $p$ independently. We show that for every constant $p>3/4$, the largest antichain among these subsets is also given by a middle layer, with probability tending to $1$ as $n$ tends to infinity. This $3/4$ is best possible, and we also characterize the largest antichains for every constant $p>1/2$. Our proof is based on some new variations of Sapozhenko's graph container method.
\medskip


\end{abstract}

\section{Introduction}\label{sec::intro}

An \emph{antichain} in a poset $P$ is a collection of pairwise incomparable elements of $P$, and the \emph{width} of $P$, denoted by $w(P)$, is the size of the largest antichain in $P$. The \emph{Boolean lattice} $\P(n)$ is the poset consisting of all subsets of $[n] := \{1,\dots, n\}$ partially ordered by containment. The Boolean lattice is naturally partitioned into antichains which we call the \emph{layers} of the Boolean lattice, denoted by $\binom{[n]}{k}$ for $0 \leq k \leq n$, which are all subsets of $[n]$ of size $k$. In 1928, Sperner~\cite{Sper} proved that the width of the Boolean lattice is attained by a layer of the Boolean lattice --- in particular, the \emph{middle layer(s)} $\binom{[n]}{\floor{n/2}}$ and $\binom{[n]}{\ceil{n/2}}$. 

A trend in modern combinatorics is the study of random versions of classic results, e.g.~\cite{BMS,CG,ST,Scha} (see~\cite{RandomExtremal} for a survey). For a set $A$ and probability $0 \leq p \leq 1$, we let $A_p$ denote a random subset of $A$, where each element of $A$ is included independently with probability $p$. Typically, $A$, and possibly $p$, depends on some counting parameter $n$. We say that $A_p$ satisfies a property $P$ \emph{with high probability (w.h.p.)} if the probability $A_p$ satisfies $P$ tends to $1$ as $n$ tends to infinity. A \emph{(coarse) threshold} $p_0 = p_0(n)$ for $P$ is a function such that for $p = \omega(p_0)$, $A_p$ satisfies $P$ w.h.p., and for $p = o(p_0)$, $A_p$ does not satisfy $P$ w.h.p. A threshold is \emph{sharp} if for every constant $\ep>0$, we have that $p>(1+\ep)p_0$ implies $A_p$ satisfies $P$ w.h.p., and $p<(1-\ep)p_0$ implies $A_p$ does not satisfy $P$ w.h.p.

We are interested in $w(\P(n)_p)$, the size of the largest antichain in a random induced subposet of the Boolean lattice, and the structure of the maximum size antichains. Answering a question of Erd\H{o}s, R\'enyi~\cite{Renyi} determined the threshold for $X = \P(n)_p$ to satisfy $w(X) = |X|$, that is, the threshold for $\P(n)_p$ to be an antichain. Kohayakawa and Kreuter~\cite{KK} determined $\lim_{n\to\infty} w(\P(n)_p)/|\P(n)_p|$ for all $p=p(n)$ w.h.p. When this limit is strictly between $0$ and $1$, they also showed that there is an antichain of size $(1-o(1))w(\P(n)_p)$ contained in some (depending on $p$) consecutive layers centered on the middle layer(s) w.h.p. This characterization is generalized by Osthus~\cite{Osthus} to when $w(\P(n)_p)/|\P(n)_p| \to 0$: for $p = \omega((r/n)^r \log n)$, where $r=r(n)$ is an integer-valued function, we have w.h.p.\ that
\[ w(\P(n)_p) \leq (1 + O(1/r) + o(1)) p \sum_{-r/2\leq j<r/2} \binom{n}{\floor{n/2}+j} .\]
A lower bound given by the maximum antichain in some consecutive layers shows that this result determines $w(\P(n)_p)$ asymptotically w.h.p.\ when $r \to \infty$, that is, when $p$ decays faster than polynomially. Balogh, Mycroft, and Treglown~\cite{BMT} asymptotically determined $w(\P(n)_p)$ w.h.p.\ when $p$ decays polynomially in $n$ using a two-phase graph container method. Namely they showed that $p_0 = 1/n^r$, for constant integer $r$, is a coarse threshold for
\[ w(\P(n)_p) = (1+o(1)) p \sum_{-r/2 \leq j < r/2} \binom{n}{\floor{n/2}+j} = (1+o(1)) p r \binom{n}{\floor{n/2}} .\]
The case $r=1$ solves a problem posed by Kohayakawa and Kreuter (in a draft of~\cite{KK}), as for $p = \omega(1/n)$, the width of $\P(n)_p$ is asymptotically equal to the size of a middle layer (this special case was proven independently by Collares Neto and Morris~\cite{CNM}). Kohayakawa and Kreuter also wondered about the threshold for the width to be exactly equal to the size of a middle layer. Hamm and Kahn~\cite{HK2} proved, as a consequence of a proof of a similar random version of the Erd\H{o}s--Ko--Rado Theorem, that there exists a constant $p<1$ such that for $X = \P(n)_p$,
\[ w(X) = \max\left\{ X \cap \binom{[n]}{\floor{n/2}}, X \cap \binom{[n]}{\ceil{n/2}} \right\} \]
w.h.p., calling this a ``more literal [random] counterpart of Sperner's Theorem.'' Our main result is that this statement holds w.h.p.\ for every constant $p>3/4$, which is best possible.
\begin{theorem}\label{thm::main}
For every constant $p>3/4$, the largest antichain in $\P(n)_p$ is (one of) the middle layer(s), with high probability.
\end{theorem}

Call a set $v \in \binom{[n]}{\ceil{n/2}+1}$ \emph{isolated} if $v$ appears in $X=\P(n)_p$ but no subset of $v$ of size $\ceil{n/2}$ appears in $X$. To see that $p_0 = 3/4$ in Theorem~\ref{thm::main} is best possible, note that the probability that $v$ is isolated is
\[ p_0 (1-p_0)^{\ceil{n/2}+1} = 3 \cdot 4^{-\ceil{n/2}-2} .\]
Since there are $\binom{n}{\ceil{n/2}+1} \approx 2^n$ choices for $v$, an easy second moment calculation shows that w.h.p.\ isolated $v$ exist for every constant $p<3/4$, and hence $X \cap \binom{[n]}{\ceil{n/2}}$ is not a maximal antichain in $X = \P(n)_p$ w.h.p.

Our proof of Theorem~\ref{thm::main} combines Sapozhenko's graph containers~\cite{Sap} with Hamm and Kahn's proof~\cite{HK2}, which itself is inspired by Sapozhenko's containers. An early combinatorial problem about posets was posed by Dedekind~\cite{Dede} in 1897, when he asked for the number of antichains in $\P(n)$. Korshunov~\cite{Korsh} solved this problem asymptotically, and Sapozhenko developed, generalized, and simplified Korshunov's arguments into a powerful asymptotic enumeration technique. Sapozhenko's containers also provided a simplified method of asymptotically counting the number of independent sets in the hypercube, originally due to Korshunov and Sapozhenko~\cite{KS}; see~\cite{Galvin} for an excellent exposition in English. In recent years, the containers have proven quite versatile, being employed in various statistical physics models. See Galvin and Kahn~\cite{GK} for an early example of this, see Jenssen and Keevash~\cite{JK} for a recent and encompassing paper, and see Jenssen and Perkins~\cite{JP} for a gentler introduction to the statistical physics models and tools. We also mention that Balogh, Garcia, and Li~\cite{BGL} applied Sapozhenko's containers to the same graph that we work with in our main technical theorem, Theorem~\ref{thm::2-layers}, although our analysis is very different.

We actually prove a more precise result, akin to a `hitting time' statement. Expanding the above definition, for $v \in \binom{[n]}{k}$ with $k \geq \ceil{n/2}$ (resp.\ $k \leq \floor{n/2}$), we say that $v$ is \emph{nearly isolated} if $v$ is in $X=\P(n)_p$, but at most one subset (resp.\ superset) of $v$ of size $k-1$ (resp.\ $k+1$) is in $X$.
\begin{theorem}\label{thm::hit}
For every constant $p>1/2$, there exists $k \in \{\floor{n/2},\ceil{n/2}\}$ such that every largest antichain in $X = \P(n)_p$ consists of sets in $X \cap \binom{[n]}{k}$ together with nearly isolated sets in $\binom{[n]}{k-1} \cup \binom{[n]}{k+1}$, with high probability.
\end{theorem}
Theorem~\ref{thm::hit} implies Theorem~\ref{thm::main}, as for constant $p>3/4$, there are no nearly isolated sets, by an easy first moment calculation. The $p_0 = 1/2$ in Theorem~\ref{thm::hit} is also best possible, as for every constant $p<1/2$, we expect to see the following structure for $k = \ceil{n/2}$: sets $u,v \in \binom{[n]}{k+1} \cap X$ that share a subset $u'$ in $\binom{[n]}{k} \cap X$, and at most one set $v' \neq u'$ is a subset of either $u$ or $v$ appearing in $X \cap \binom{[n]}{k}$, where $X = \P(n)_p$. That is, one of $u$ or $v$ is not nearly isolated, but $u$ and $v$ together act like two nearly isolated sets, in the sense that $u$ and $v$ together with $\big(\binom{[n]}{k} \cap X\big) \setminus \{u',v'\}$ form an antichain. See Figure~\ref{fig::defects}. To see that this structure appears below $p_0 = 1/2$, note that there are roughly $(n/2)^3 \binom{n}{n/2} \approx 2^n$ choices for $u$, $v$, $u'$, and $v'$, and the probability that they form the desired structure in our random subfamily is roughly $p_0^4 (1-p_0)^{2(n/2)-3} = 2^{n+1}$.

\tikzset{vtx/.style={inner sep=1pt, outer sep=0pt, circle, fill=black, draw=black}}

\begin{figure}
\centering
\begin{tikzpicture}
\def\x{1.2}
\draw[rounded corners] (-2*\x,-.3) rectangle (10.5*\x,1);
\draw (-2*\x-.5,.3) node{$\binom{[n]}{k+1}$};
\draw[rounded corners] (-2*\x,-.6) rectangle (10.5*\x,-2);
\draw (-2*\x-.5,-1.3) node{$\binom{[n]}{k}$};
\draw (2*\x, 1.3) -- (2*\x, -2.3);
\begin{scope}[shift={(0,0)}]
	\draw (0,0) node[vtx]{} node[above]{$v \in X$};

	\draw (0,0) -- (-3*\x/2,-1) node[vtx]{} node[below]{$\not\in X$};
	\draw (0,0) -- (-1*\x/2,-1) node[vtx]{} node[below]{$\not\in X$};
	\draw (0,0) -- ( 1*\x/2,-1) node[vtx]{} node[below]{$\not\in X$};
	\draw (0,0) -- ( 3*\x/2,-1) node[vtx]{} node[below]{$\not\in X$};
\end{scope}
\begin{scope}[shift={(4*\x,0)}]
	\draw (0,0) node[vtx]{} node[above]{$v \in X$};
	\draw (4*\x,0) node[vtx]{} node[above]{$u \in X$};

	\draw (0,0) -- (-3*\x/2,-1) node[vtx]{} node[below]{$\not\in X$};
	\draw (0,0) -- (-1*\x/2,-1) node[vtx]{} node[below]{$\not\in X$};
	\draw (0,0) -- ( 1*\x/2,-1) node[vtx]{} node[below]{$\not\in X$};
	\draw (0,0) -- ( 4*\x/2,-1) node[vtx]{} node[below]{$u'\in X$} -- (4*\x,0);
	\draw (4*\x,0) -- ( 7*\x/2,-1) node[vtx]{} node[below]{$v'\in X$};
	\draw (4*\x,0) -- ( 10*\x/2,-1) node[vtx]{} node[below]{$\not\in X$};
	\draw (4*\x,0) -- ( 12*\x/2,-1) node[vtx]{} node[below]{$\not\in X$};
\end{scope}
\end{tikzpicture}
\caption{$X = \P(n)_p$ and the edges drawn represent the subset relation. On the left for $k = \ceil{n/2}$, $v$ is an isolated set as defined after Theorem~\ref{thm::main}. On the right for $k \geq \ceil{n/2}$, $v$ is a nearly isolated set as defined before Theorem~\ref{thm::hit}. Both are likely to appear for $p<3/4$ but unlikely for $p>3/4$. While $u$ is not nearly isolated, $u$ and $v$ together act as two nearly isolated sets, as described after Theorem~\ref{thm::hit}. This structure is likely to appear for $p<1/2$ but unlikely for $p>1/2$.}
\label{fig::defects}
\end{figure}
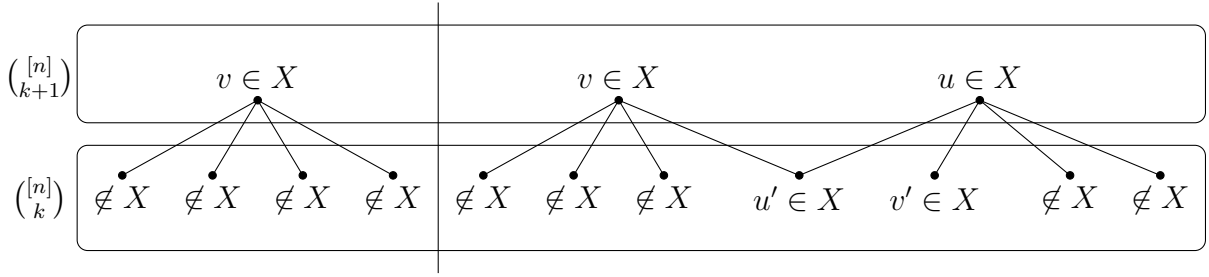

Theorem~\ref{thm::hit} gives a precise description of the maximum antichains of $\P(n)_p$ for constant $p>1/2$. It is easy to conjecture a precise description for smaller, but constant, $p$, as the maximum antichains should consist almost entirely of a whole layer together with `defects' like nearly isolated sets and the structures described in the previous paragraph. We conjecture that this roughly holds for $p$ above the $1/n$ threshold, which we recall is the threshold for the width to be asymptotically equal to the size of a middle layer \cite{BMT,CNM}.
\begin{conjecture}\label{conj::main}
For $p = \omega(1/n)$, the maximum antichains of $\P(n)_p$ are contained in three consecutive layers, with high probability.
\end{conjecture}
If true, the bound on $p$ is tight, as for $p = C/n$ with constant $C$, w.h.p.\ we see some sets $v \in \binom{[n]}{\ceil{n/2}+2} \cap X$ such that removing any one or two elements of $v$ results in a set which is not in $X = \P(n)_p$.

An \emph{$\ell$-chain} in a poset $P$ is $\ell$ elements $A_1, \dots, A_\ell$ such that $A_1 \leq \cdots \leq A_\ell$ in $P$. An antichain is a collection of elements of $P$ which contains no $2$-chain. Erd\H{o}s~\cite{Erdos} generalized Sperner's Theorem, showing that the largest subsets of $\P(n)$ with no $\ell$-chain are some $\ell-1$ layers, in particular, the $\ell-1$ consecutive layers centered about a middle layer. It would be interesting to generalize Theorem~\ref{thm::main} from antichains to collections with no $\ell$-chain, for constant $\ell$. We expect the threshold to be the same.

Theorems~\ref{thm::main} and~\ref{thm::hit} follow from the following theorem, where we take $n$ odd and only consider antichains in the middle two layers. To state the theorem, we need some definitions. Let $G$ be a bipartite graph with $Y$ as one side of the bipartition. We say $A \subseteq Y$ is \emph{closed} if there is no $A' \supsetneq A$ such that $N(A') = N(A)$. We say that $A \subseteq Y$ is \emph{2-linked} if there is no $B \subsetneq A$ such that $N(B)$ is disjoint from $N(A \setminus B)$; in other words, $A \cup N(A)$ induces a connected subgraph of $G$. We denote by $M$ the ``middle two layers'' graph: the bipartite graph with parts $L_k = \binom{[2k-1]}{k}$ and $L_{k-1} = \binom{[2k-1]}{k-1}$ and edges given by the containment relation.

\begin{theorem}\label{thm::2-layers}
Fix $\ell \in \{1,2\}$. Let $p > 1-2^{-2/\ell}$ be constant, and let $X = V(M)_p$. Then there exists a constant $c>0$ such that the following holds. With probability $1-o(1/k)$, every closed, 2-linked, $A \subseteq L_k$ (or $A \subseteq L_{k-1}$) with $\ell \leq |A|$ and $|N(A)| \geq (1+1/2k) |A|$ satisfies $|N(A) \cap X| - |A \cap X| \geq c(|N(A)|-|A|)$.
\end{theorem}

Note that $M$ is symmetric with respect to its bipartition, so the statement for $A \subseteq L_{k-1}$ easily follows from the statement for $A \subseteq L_k$. The proof of Theorem~\ref{thm::2-layers} is easy for small $A$ (see Section~\ref{sec::large-delta}), so the main difficulty of the proof is handling large $A$; see Section~\ref{sec::intro:idea} for the proof idea. Note that for $\ell = 1$, the condition on $p$ is $p>3/4$, reminiscent of Theorem~\ref{thm::main}, while for $\ell = 2$, the condition is $p>1/2$, reminiscent of Theorem~\ref{thm::hit}. Indeed, for the same reasons as for those theorems, the condition on $p$ in Theorem~\ref{thm::2-layers} would be best possible for all constant $\ell$ if true, and the condition is really only necessary to get the conclusion for small $A$. We conjecture that Theorem~\ref{thm::2-layers} is true in the above form for all constant $\ell$, but we do not know how to prove such a statement. We show that Theorem~\ref{thm::2-layers} implies Theorems~\ref{thm::main} and~\ref{thm::hit} in Section~\ref{sec::reduction}.

Theorem~\ref{thm::2-layers} can be viewed as proving a Hall-type condition for a random induced subgraph of $M$. A natural approach is to compute the probability that $|N(A) \cap X| - |A \cap X|$ is small for a given $A$ and then take a union bound over all $A$. Computing this probability can be done easily and accurately with the Chernoff bound, using only $|A|$ and $|N(A)|$. Sapozhenko's graph containers~\cite{Sap} are precisely the tool to count the number of $A$ with $|A|$ and $|N(A)|$ given.\footnote{Sapozhenko's graph containers are also very general in that they require very little information about the graph. Usually all that is needed is almost-regularity and vertex-isoperimetry. The techniques can even be applied to non-bipartite graphs, as was done by Hamm and Kahn~\cite{HK2} to the Kneser graph.} Unfortunately, the bounds on the number of such $A$ are too weak to make the union bound work. What we must do instead is dive into the construction of the containers to make the union bound more efficient. Hamm and Kahn~\cite{HK2} do this to essentially prove Theorem~\ref{thm::2-layers} with $p=1-\ep$ for some small constant $\ep>0$, but with an important difference: they looked at the graph formed by the layers $\binom{[2k]}{k}$ and $\binom{[2k]}{k-1}$. We extend their proof when the graph is $M$. Unfortunately, our methods do not work for the graph formed by $\binom{[2k]}{k}$ and $\binom{[2k]}{k-1}$, because the degrees of vertices in $\binom{[2k]}{k}$ are smaller than the degrees of vertices in $\binom{[2k]}{k-1}$; see the remark at the end of Section~\ref{sec::strong}.

We mention a variant of Theorem~\ref{thm::2-layers} that is better understood. If we keep the edges of $M$ with probability $p$ instead of the vertices, then we obtain a variant of Theorem~\ref{thm::2-layers} which we can prove holds for much smaller $p$ (and correspondingly larger $\ell$). In a paper with Luo~\cite{BKL}, we prove a theorem like this for the graph formed by the layers $\binom{[2k]}{k}$ and $\binom{[2k]}{k-1}$. Analogous to~\cite{HK2}, this yields a random version of the Erd\H{o}s--Ko--Rado Theorem. See~\cite{BKL,HK2} for the relation between the Erd\H{o}s--Ko--Rado Theorem and the graph formed by the layers $\binom{[2k]}{k}$ and $\binom{[2k]}{k-1}$.

Finally, we mention that the length of $\P(n)_p$ has also been studied. The \emph{length} of a poset is the size of the longest chain. In fact, Kohayakawa, Kreuter, and Osthus~\cite{KKO} determined some rather precise information about the length of $\P(n)_p$. This problem seems to be rather different than determining $w(\P(n)_p)$, using much more of the structure of $\P(n)$ than we do here.

\subsection{Proof idea}\label{sec::intro:idea}

Our proof follows Hamm and Kahn's proof~\cite{HK2}, presented in our notation, with an important detour in Section 5.

To prove Theorem~\ref{thm::2-layers}, we should know how $|N(A)|$ compares to $|A|$. For $M$, the ``middle two layers'' graph, this is given by the Kruskal--Katona theorem, a weaker form of which is due to Lov\'asz. For real $z$ we define $\binom{z}{k} = \frac{z (z-1) \cdot \ldots \cdot (z-k+1)}{k!}$, and we let $\partial_* A = \{v \setminus \{i\} : i \in v \in A\}$ denote the \emph{lower shadow} of $A$.
\begin{theorem}[Kruskal \cite{Kruskal}, Katona \cite{Katona}, Lov\'asz (Problem 13.31(b) of \cite{Lovasz})]\label{thm::KK}
Let $A \subseteq \binom{[n]}{k}$ be of size $\binom{z}{k}$ for $z \geq k$. Then $|\partial_* A| \geq \binom{z}{k-1}$.
\end{theorem}
Using this, we derive a convenient expression for how much sets in $M$ `expand'; see Proposition 2.3 of~\cite{HK2} for a similar inequality. Recall that $M$ is the graph between $L_k = \binom{[2k-1]}{k}$ and $L_{k-1} = \binom{[2k-1]}{k-1}$ given by containment. All logarithms are natural unless otherwise indicated.
\begin{lemma}\label{lem::iso}
For $A \subseteq L_k$ with $|A| = a \geq 1$, we have
\[ t := |N(A)| - |A| \geq a \frac{1}{k} \log \frac{\binom{2k-1}{k}}{a} .\]
\end{lemma}
\begin{proof}
Let $|A| = \binom{2k-1-y}{k}$ where $0 \leq y \leq k-1$. By Theorem~\ref{thm::KK}, we have
\[ |N(A)| - |A| \geq \binom{2k-1-y}{k-1} - \binom{2k-1-y}{k} = \frac{y}{k-y} \binom{2k-1-y}{k} .\]
It suffices to show that $f(y) \geq 0$, where
\[ f(y) = \frac{y}{k-y} - \frac{1}{k} \log \frac{\binom{2k-1}{k}}{\binom{2k-1-y}{k}} .\]
Indeed, $f(0) = 0$, and
\[ f'(y) = \frac{k}{(k-y)^2} - \frac{1}{k} \sum_{i=0}^{k-1} \frac{1}{2k-1-i-y} \geq \frac{k}{(k-y)^2} - \frac{1}{k-y} \geq 0 .\qedhere\]
\end{proof}

The rough idea of the proof of Theorem~\ref{thm::2-layers} is to find $S$ and $F$ such that $A \subseteq S$, $F \subseteq N(A)$, $|S| < |F| - \ell$ for some $\ell$, and $|S \cap X| = p|S| \pm p\ell/2$, $|F \cap X| = p|F| \pm p\ell/2$ with (very) high probability. Then we have
\[ |A \cap X| \leq |S \cap X| \leq p|S| + p\ell/2 \leq p|F| - p\ell/2 \leq |F \cap X| \leq |N(A) \cap X| .\]
For each $A$, we might show the existence of $S$ and $F$ with the first three properties, and then show that across all $A$, there are actually few choices for $S$ and $F$ --- this is the heart of Sapozhenko's graph container method~\cite{Sap}. Using standard concentration arguments we can then get that $|S \cap X|$ and $|F \cap X|$ are concentrated about their means with (very) high probability for all choices of $S$ and $F$.

Unfortunately, the above scenario does not happen perfectly. In fact, all we can guarantee is $|S| \leq |F| + \ell$, so we must also use that $|(S \setminus A) \cap X|$ and $|(N(A) \setminus F) \cap X|$ are likely large in order to get the above inequality chain to work out. This also allows us to get a stronger lower bound on $|N(A) \cap X| - |A \cap X|$.

We break the proof into three phases. In the first phase, we obtain `weak' versions of $S$ and $F$ described in the above sketch; see Lemma~\ref{lem::weak}, which is proved in Section~\ref{sec::weak}. In the second phase, we strengthen these weak versions; see Lemma~\ref{lem::strong}, which is proved in Section~\ref{sec::strong}. In the final phase, we give the argument showing $|N(A) \cap X| - |A \cap X|$ is large for all $A$ with (very) high probability, which we prove in Section~\ref{sec::use-containers}.

We first state the two container lemmas formally (without asymptotic notation), and then we develop some notions which help with the intuition behind the lemmas and the mechanics of their proofs. Let
\begin{equation}\label{eq::defA}
\A_t = \left\{ A \subseteq L_k : A \text{ is 2-linked}, t = |N(A)|-|A| \geq |A|/2k \right\}
\end{equation}
be the set of all subsets of $L_k$ with a given expansion property. As we shall see, $t$ provides a nice scale for measuring various aspects of the containers.
\begin{lemma}\label{lem::weak}[Weak container lemma]
Let $p$ be a constant and let $X = V(M)_p$. There exists a global constant $K$, such that for every $\ep>0$, there exists $k_0$ such that for all $k \geq k_0$, the following holds. For every $t \geq k \log^7 k$, there exists a collection $\mathcal{W}_t'$ of pairs $(S',F')$ with $|\mathcal{W}_t'| \leq 2^{\ep t}$ such that the following holds with probability at least $1-\exp(-k/\ep)$: for all $A \in \A_t$, there exists $(S',F') \in \mathcal{W}_t'$ with
\begin{itemize}
\item $A \subseteq S'$,
\item $F' \subseteq N(A)$,
\item $|S'| \leq |F'| + Kt$,
\item $||S' \cap X|-p|S'|| \leq \ep t$,
\item $||F' \cap X|-p|F'|| \leq \ep t$.
\end{itemize}
\end{lemma}
Lemma~\ref{lem::weak} gives a `small' collection $\mathcal{W}_t'$ of `containers' $(S',F')$ which `cover' $\A_t$. Furthermore, these $S'$ and $F'$ are `concentrated' in the sense that their size in the random subgraph of $M$ is close to their expected size. We are only able to obtain $\mathcal{W}_t'$ for sufficiently large $t$ in terms of $k$, but fortunately the proof of Theorem~\ref{thm::2-layers} for small $t$ is relatively easy; see Section~\ref{sec::large-delta}.

Hamm and Kahn~\cite{HK2} essentially proved Lemma~\ref{lem::weak} in a slightly different graph. Lemma~\ref{lem::weak}, through the argument given in Section~\ref{sec::use-containers}, yields Theorem~\ref{thm::main} with $p = 1-\ep_K$, where $\ep_K$ is a constant depending only on the $K$ in Lemma~\ref{lem::weak}. We present their proof not just for completeness, but also to phrase it in a language which we feel brings the formal proof closer to heuristics. Our other main contribution to the theory of graph containers is the transformation of `concentrated weak containers' into `concentrated strong containers,' which is essentially taking $K$ arbitrarily small in Lemma~\ref{lem::weak}, as stated in the next lemma.
\begin{lemma}\label{lem::strong}[Strong container lemma]
Let $p$ be a constant and let $X = V(M)_p$. For every $\ep > 0$, there exists $k_0$ such that for all $k \geq k_0$ the following holds. For every $t \geq k \log^7 k$, there exists a collection $\mathcal{W}_t$ of pairs $(S,F)$ with $|\mathcal{W}_t| \leq 2^{\ep t}$ such that the following holds with probability at least $1-\exp(-k/\ep)$: for all $A \in \A_t$, there exists $(S,F) \in \mathcal{W}_t$ with
\begin{itemize}
\item $A \subseteq S$,
\item $F \subseteq N(A)$,
\item $|S| \leq |F| + \ep t$
\item $||S \cap X|-p|S|| \leq \ep t$,
\item $||F \cap X|-p|F|| \leq \ep t$.
\end{itemize}
\end{lemma}
The proof of Lemma~\ref{lem::strong} essentially follows a proof of Sapozhenko~\cite{Sap}, as described by Galvin (Lemma~5.5 in~\cite{Galvin}), except that we must modify it to ensure these `concentration' conditions.

\subsection{Terminology}\label{sec::intro:terms}

We heavily use asymptotic notation as it simplifies our calculations and definitions. Recall that $f(n) = O(g(n))$, or equivalently $g(n) = \Omega(f(n))$, if there exists a constant $K$ such that $f(n) \leq Kg(n)$ for all sufficiently large $n$. We say $f(n) = o(g(n))$, or equivalently $g(n) = \omega(f(n))$, if $f(n)/g(n)$ tends to $0$ as $n$ tends to infinity. All uses of asymptotic notation are with $n$ or $k$ tending to infinity, as appropriate, and hence some inequalities only hold for $n$ or $k$ sufficiently large. We will often abuse notation and write, for example, $f(n) \leq O(g(n))$ in an inequality chain, to mean $f(n) = O(g(n))$. Most uses of $O(\cdot)$ or $o(\cdot)$ in this paper could be replaced by a sufficiently large or small constant, respectively. However, we prefer the asymptotic notation because it hides irrelevant arithmetic on these expressions. 

We introduce some nonstandard terminology which makes the statements and proofs of Lemmas~\ref{lem::weak} and~\ref{lem::strong} more digestible. The way we produce the containers of Lemmas~\ref{lem::weak} and~\ref{lem::strong} is by taking an arbitrary $A \in \A_t$ and creating several derivative sets from it. We need a convenient way to count the number of choices for these derivative sets across all $A \in \A_t$.
\begin{definition}
Let $\Y = \{Y_A : A \in \A_t\}$, where $Y_A$ is an object which depends on $A$. The \emph{cost} of $\Y$ is $\log_2|\Y|$. When the dependence of $Y_A$ on $A$ is understood, we often leave this collection $\Y$ implicit and write $\cost(Y)$ for the cost of $\Y$. We let $\cost((Y_1, \dots, Y_\ell)) = \cost(Y_1, \dots, Y_\ell)$ for ease of notation.

We denote by $\cost(Y|Z)$ the cost of $Y$ given a choice of $Z$. More formally, for $\Z = \{Z_A : A \in \A_t\}$, $\cost(Y|Z) = \log_2\left|\left\{ Y_A : A \in \A_t, Z_A = Z \right\}\right|$ depends on the choice of $Z$.
\end{definition}
Intuitively, the cost of an object is the $\log_2$ of the number of choices for that object. For example, the cost of a subset $Y$ of a given set $Z$ is at most $|Z|$, while the cost of a subset $Y$ (of $Z$) of size at most a given $\ell$ is at most
\[ \log_2 \binom{|Z|}{\leq \ell} = \log_2\left( \binom{|Z|}{0} + \binom{|Z|}{1} + \cdots + \binom{|Z|}{\ell} \right) \leq \ell \log_2 \frac{e|Z|}{\ell} ,\]
a fact which we will use often without mention. We give some easy facts about cost in the next proposition, which we note are analogous to facts about binary entropy.

\begin{proposition}\label{prop::costfacts}
For any $Y$ and $Z$, we have $\cost(Y,Z) \leq \cost(Z) + \max_Z \cost(Y|Z)$. If $Y$ is determined by $Z$, that is, if there is a surjective function $f$ from $\Z = \{Z_A : A \in \A_t\}$ to $\Y = \{Y_A : A \in \A_t\}$ with $f(Z_A) = Y_A$ for all $A \in \A_t$, then $\cost(Y) \leq \cost(Z)$.
\end{proposition}

We also need a convenient way to discuss concentration in the random subgraph of $M$.
\begin{definition}
A collection $\Y$ of subsets of $V(M)$ is \emph{$(\ep,t,p)$-concentrated} if $||Y \cap X|-p|Y|| \leq \ep t$ for all $Y \in \Y$ with probability at least $1-\exp(-k/\ep)$, where $X = V(M)_p$.
\end{definition}

The following are easy consequences of this definition which we use extensively.
\begin{proposition}\label{prop::autoconc}
If $\Y$ is a collection of subsets of $V(M)$ and $\max_{Y \in \Y} |Y| \leq \ep t$, then $\Y$ is $(\ep,t,p)$-concentrated for every $p$.
\end{proposition}

\begin{proposition}\label{prop::unionconc}
Let $\Y = \{Y_A : A \in \A_t\}$ and $\Z = \{Z_A : A \in \A_t\}$ be $(\ep,t,p)$-concentrated. If $Y_A \cap Z_A = \emptyset$ for all $A \in \A_t$, then $\{Y_A \cup Z_A : A \in \A_t\}$ is $(2\ep,t,p)$-concentrated. If $Y_A \subseteq Z_A$ for all $A \in \A_t$, then $\{Z_A \setminus Y_A : A \in \A_t\}$ is $(2\ep,t,p)$-concentrated.
\end{proposition}

\begin{proof}
Let $X = V(M)_p$. For the first claim, observe that since $Y_A$ and $Z_A$ are disjoint, we have
\[ ||(Y_A \cup Z_A) \cap X| - p|Y_A \cup Z_A|| = ||Y_A \cap X| + |Z_A \cap X| - p|Y_A| - p|Z_A|| \]
\[ \leq ||Y_A \cap X| - p|Y_A|| + ||Z_A \cap X| - p|Z_A|| \leq 2\ep t .\]
For the second claim, since $Y_A \subseteq Z_A$, we have
\[ ||(Z_A \setminus Y_A) \cap X| - p|Z_A \setminus Y_A|| = ||Z_A \cap X| - |Y_A \cap X| - p|Z_A| + p|Y_A|| \]
\[ \leq ||Z_A \cap X| - p|Z_A|| + ||Y_A \cap X| - p|Y_A|| \leq 2\ep t . \qedhere\]
\end{proof}

\begin{proposition}\label{prop::Chern}
Let $\Y$ be a collection of subsets of $V(M)$. If
\[ \cost(Y) = \log_2|\Y| \leq \frac{\ep^2}{2p+1} \cdot \frac{t^2}{\max_{Y \in \Y} |Y|} \]
and
\[ \frac{5k}{\ep} \leq \frac{\ep^2}{2p+1} \cdot \frac{t^2}{\max_{Y \in \Y} |Y|} ,\]
then $\Y$ is $(\ep,t,p)$-concentrated. 
\end{proposition}

\begin{proof}
Let $X = V(M)_p$, let $m = \max_{Y\in\Y} |Y|$, and let $Y \in \Y$. If $|Y| \leq \ep t$, then automatically we have $||Y \cap X| - p|Y|| \leq \ep t$. Otherwise, the standard Chernoff bound states that the probability that $|Y \cap X| - p|Y| \leq -\ep t$ is at most
\[ \exp\left( - \frac{(\ep t)^2}{2p|Y|} \right) \leq \exp\left( - \frac{(\ep t)^2}{(2p+1)m} \right) ,\]
and the probability that $|Y \cap X| - p|Y| \geq \ep t$ is at most
\[ \exp\left( - \frac{(\ep t)^2}{2p|Y|+\ep t} \right) \leq \exp\left( - \frac{(\ep t)^2}{(2p+1)m} \right) \]
using that $|Y| \geq \ep t$. Taking a union bound over all $Y \in \Y$ with $|Y| \geq \ep t$, we have that the probability that $||Y \cap X| - p|Y|| \leq \ep t$ for all $Y \in \Y$ is at least
\[ 1 - 2 \exp\left( \cost(Y) \log(2) - \frac{(\ep t)^2}{(2p+1)m} \right) \geq 1 - 2 \exp\left( - \frac{(\ep t)^2}{4(2p+1)m} \right) \geq 1 - \exp\left( - k/\ep \right) . \qedhere\]
\end{proof}

The conditions of Proposition~\ref{prop::Chern} ensure that the cost of $Y$ is small and that $Y$ is not too large as to make the desired concentration unlikely. These may be stated in a slightly obtuse manner, but it is convenient to just compute $\cost(Y)$ and $t^2 / \max_{Y\in\Y} |Y|$ and confirm that the latter is much greater than the former and $k$.

In our proofs, typically $t=|N(A)|-|A|$ is given, $\mathcal{Y}$ is implicitly understood to be the collection of all choices of some variable $Y$ which is derived from $A \in \A_t$, $p$ is understood to be a constant greater than $1/2$, and $\ep = o(1)$, which the reader may think of as a sufficiently small constant. To minimize the minutiae, we often omit $\ep$, $t$, and $p$ and just say `$Y$ is concentrated' for short. This allows us to not deal with the changing $\ep$, say in Proposition~\ref{prop::unionconc}, and it simplifies the conditions of Proposition~\ref{prop::Chern} to
\[ \frac{t^2}{|Y|} \geq \max\left\{\omega(\cost(Y)), \omega(k)\right\} \]
for all $Y$.

The final definition we introduce concerns the two types of containers we find.
\begin{definition}
We say that $(S', F')$ is a \emph{weak container} for $A \in \A_t$ if $A \subseteq S'$, $F' \subseteq N(A)$, $|S' \setminus A| = O(t)$, and $|N(A) \setminus F'| = O(t)$. We say that $(S,F)$ is a \emph{strong container} for $A \in \A_t$ if $(S,F)$ is a weak container and also satisfies $|S| \leq |F| + o(t)$.
\end{definition}
Note that, assuming $A \subseteq S'$ and $F' \subseteq N(A)$, the conditions $|S' \setminus A| = O(t)$ and $|N(A) \setminus F'| = O(t)$ together are equivalent to $|S'| \leq |F'| + O(t)$. This explains the way in which strong containers are a stronger version of weak containers.

Now we can conveniently summarize the two container lemmas. Lemma~\ref{lem::weak} states that there exists concentrated weak containers for all $A \in \A_t$ with $t \geq k\log^7 k$ at cost $o(t)$, where a container $(S,F)$ is concentrated if both $S$ and $F$ are concentrated. Lemma~\ref{lem::strong} states that there exists concentrated strong containers for all $A \in \A_t$ with $t \geq k\log^7 k$ at cost $o(t)$.

We prove Lemma~\ref{lem::weak} in Section~\ref{sec::weak} and Lemma~\ref{lem::strong} in Section~\ref{sec::strong}. In Section~\ref{sec::use-containers}, we use these lemmas to prove Theorem~\ref{thm::2-layers}.

\section{Theorem~\ref{thm::2-layers} implies Theorem~\ref{thm::hit}}\label{sec::reduction}

Hamm and Kahn~\cite{HK2} give a sketch of how Theorem~\ref{thm::2-layers} implies Theorem~\ref{thm::main}. More work is required to deduce Theorem~\ref{thm::hit}. We first generalize Theorem~\ref{thm::2-layers} to apply to any pair of consecutive layers. For $A \subseteq \P(n)$, let the \emph{lower shadow} of $A$ be $\partial_* A = \{v \setminus \{i\} : i \in v \in A\}$ and the \emph{upper shadow} of $A$ be $\partial^* A = \{v \cup \{i\} : i \not\in v \in A\}$. When $A \subseteq \binom{[n]}{k}$ with $k \geq \ceil{n/2}$, let $\partial A = \partial_* A$, and when $k \leq \floor{n/2}$, let $\partial A = \partial^* A$, which we simply call the \emph{shadow} of $A$. This definition is non-standard but helps simplify our statements, so we use only when it will not cause confusion. A simple consequence of the Kruskal--Katona Theorem is that $\partial A$ is much bigger than $A$ for $A$ outside of the two middle layers.

\begin{proposition}\label{prop::expoutmid}
For $A \subseteq \binom{[n]}{k}$ and $k \in [n] \setminus \{\ceil{n/2},\floor{n/2}\}$, we have $|\partial A| \geq \left(1+\frac{1}{k}\right)|A|$.
\end{proposition}

\begin{proof}
By symmetry, we may assume that $k>\ceil{n/2}$. By Theorem~\ref{thm::KK}, for $|A| = \binom{x}{k}$ we have
\[ |\partial A| = |\partial_* A| \geq \binom{x}{k-1} = \frac{k}{x-k+1} \binom{x}{k} \geq \frac{k}{n-k+1} |A| \geq \frac{n/2+1}{n/2} |A| \geq \left( 1 + \frac{1}{k} \right) |A| .\qedhere\]
\end{proof}

Extending the definitions from $M$, we say that $A \subseteq \binom{[n]}{k}$ is \emph{closed} if there is no $A \subsetneq A' \subseteq \binom{[n]}{k}$ with $\partial A = \partial A'$, and we say that $A \subseteq \binom{[n]}{k}$ is \emph{2-linked} if there is no $B \subsetneq A$ such that $\partial B$ is disjoint from $\partial (A \setminus B)$.
\begin{lemma}\label{lem::2-layers}
Fix $\ell \in \{1,2\}$. Let $p > 1-2^{-2/\ell}$ be constant, and let $X = \P(n)_p$. Then there exists a constant $c>0$ such that the following holds. With high probability, for every $k \in [n] \setminus \{\ceil{n/2},\floor{n/2}\}$, every closed, 2-linked $A \subseteq \binom{[n]}{k}$ with $|A| \geq \ell$ satisfies $|\partial A \cap X| - |A \cap X| \geq c(|\partial A| - |A|) > 0$.
\end{lemma}
\begin{proof}
As the other case is symmetric to this one, we assume $k>\ceil{n/2}$. Let $A \subseteq \binom{[n]}{k}$ be closed and 2-linked with $k \in [n] \setminus \{\ceil{n/2},\floor{n/2}\}$. Viewing $A$ as a subset of $\binom{[2k-1]}{k}$, which we may do because $\partial_*$ does not depend on $n$, and noting that $|\partial A| \geq (1+1/2k)|A|$ by Proposition~\ref{prop::expoutmid}, the conclusion of Theorem~\ref{thm::2-layers} gives the desired conclusion. Because this holds with probability $1-o(1/n)$, taking a union bound over all $k$ finishes the proof.
\end{proof}

Although Theorem~\ref{thm::hit} implies Theorem~\ref{thm::main}, we give a proof of Theorem~\ref{thm::main} first as a warm-up for Theorem~\ref{thm::hit}. A \emph{2-linked component} of $A \subseteq \binom{[n]}{k}$ is a maximal 2-linked subset of $A$. For $A \subseteq \binom{[n]}{k}$, let $[A]$ denote the \emph{closure} of $A$, the largest $A'$ with $\partial A = \partial A'$.

\begin{proof}[Proof of Theorem~\ref{thm::main}]
Let $p>3/4$ be constant, and let $X = \P(n)_p$. We assume that the conclusions of Theorem~\ref{thm::2-layers} and Lemma~\ref{lem::2-layers} hold for $\ell = 1$, which happens with high probability.

Let $B$ be a maximum antichain in $X$, and let $B_\ell = B \cap \binom{[n]}{\ell}$ be the layers of $B$. Consider the nonempty layer $B_k$ of $B$ with the largest $k$. If $k > \ceil{n/2}$, then letting $B_k'$ be a 2-linked component of $B_k$, we claim that $B' = (B \setminus B_k') \cup (\partial_* B_k' \cap X)$ is a larger antichain than $B$ in $X$. Indeed, by definition, $B'$ is an antichain in $X$, and $|B'| > |B|$ because of Lemma~\ref{lem::2-layers} applied to the closure of $B_k'$. This contradicts the maximality of $B$, and hence $k \leq \ceil{n/2}$.

Similarly, the lowest nonempty layer of $B$ is at layer number $\floor{n/2}$ or above. If $n$ is even, this implies that $B$ is contained in the middle layer, as desired. If $n$ is odd, one of $[B_{\floor{n/2}}]$ or $[B_{\ceil{n/2}}]$ has size at most $\frac{1}{2} \binom{n}{\floor{n/2}}$ by Theorem~\ref{thm::KK}. Since $M$ is symmetric with respect to $L_k$ and $L_{k-1}$, we may suppose it is $B_{\ceil{n/2}}$ whose closure has size at most $\frac{1}{2} \binom{n}{\floor{n/2}}$, so by Lemma~\ref{lem::iso}, $|\partial_* B_{\ceil{n/2}}| \geq (1+\frac{1}{2\ceil{n/2}})|[B_{\ceil{n/2}}]|$. If $B_{\ceil{n/2}} \neq \emptyset$, then $B_{\floor{n/2}} \cup (\partial_* B_{\ceil{n/2}} \cap X)$ is a larger antichain than $B$ in $X$ by Theorem~\ref{thm::2-layers}, contradicting the maximality of $B$. Thus $B$ is contained within one of the middle two layers.
\end{proof}

To prove Theorem~\ref{thm::hit} using Theorem~\ref{thm::2-layers} and Lemma~\ref{lem::2-layers}, we modify the previous argument by taking special care of 2-linked components of size $1$ in the layers of the maximum antichain.

\begin{proof}[Proof of Theorem~\ref{thm::hit}]
Let $p=1/2+\ep$, where $\ep>0$ is constant, and let $X = \P(n)_p$ be such that the conclusions of Theorem~\ref{thm::2-layers} and Lemma~\ref{lem::2-layers} hold for $\ell=2$, which happens with high probability.

We first claim that with high probability the number of nearly isolated sets is at most $2^{n/2}$. Recall that $v \in \P(n)$ is nearly isolated if $v \in X$ but $|\partial \{v\} \cap X| \leq 1$. Since $|\partial\{v\}| \geq n/2$ for every $v$, the expected number of such sets is at most
\[ 2^n p \left( (1-p)^{n/2} + \frac{n}{2} p(1-p)^{n/2-1} \right) \leq 2^{n/2} \cdot O(n e^{-\ep n}) .\]
The claim follows by Markov's inequality, so we assume the claim holds for $X$.

Let $B$ be a maximum antichain in $X$, and let $B_\ell = B \cap \binom{[n]}{\ell}$ be the layers of $B$. Consider the nonempty layer $B_k$ of $B$ with the largest $k$. If $k > \ceil{n/2}$ and $B_k$ has a 2-linked component $B_k'$ of size greater than $1$, then $(B \setminus B_k') \cup (\partial_* B_k' \cap X)$ is a larger antichain than $B$ in $X$, as seen by applying Lemma~\ref{lem::2-layers} to the closure of $B_k'$. If $B_k$ has a 2-linked component $B_k'$ of size $1$ which is not a nearly isolated set, then by definition, $(B \setminus B_k') \cup (\partial_* B_k' \cap X)$ is larger antichain than $B$ in $X$. Thus if $k>\ceil{n/2}$, $B_k$ consists entirely of nearly isolated sets, and hence $|B_k| \leq 2^{n/2}$.

If $k>\ceil{n/2}+1$, we claim that
\[ (B \setminus (B_k \cup B_{k-1})) \cup ((\partial_* \partial_* B_k \cup \partial_* B_{k-1}) \cap X) \]
is a larger antichain than $B$ in $X$. If $|\partial_* B_k| \cdot \log n \leq |B_{k-1}|$, then by Proposition~\ref{prop::expoutmid}, we have
\[ |\partial_*(\partial_* B_k \cup B_{k-1})| - |[\partial_* B_k \cup B_{k-1}]| \geq \frac{1}{k} |\partial_* B_k \cup B_{k-1}| \geq \frac{\log n}{2k} |\partial_* B_k| .\]
If $|\partial_* B_k| \cdot \log n \geq |B_{k-1}|$, then $|\partial_* B_k \cup B_{k-1}| \leq (1+\log n) |\partial_* B_k| \leq (1+\log n) n 2^{n/2}$. Since $[A] \subseteq \partial^* \partial_* A$ for $A \subseteq \binom{[n]}{k}$ with $k \geq \ceil{n/2}$, we have that $|[\partial_* B_k \cup B_{k-1}]| \leq (1+\log n) n^3 2^{n/2}$, and hence by Lemma~\ref{lem::iso}, viewing $[\partial_* B_k \cup B_{k-1}]$ as a subset of $\binom{[2k-3]}{k-1}$, we have
\[ |\partial_*(\partial_* B_k \cup B_{k-1})| - |[\partial_* B_k \cup B_{k-1}]| \geq \Omega(|[\partial_* B_k \cup B_{k-1}]|) \geq \frac{\log n}{k} |\partial_* B_k| .\]
Applying Lemma~\ref{lem::2-layers} to the closure of each 2-linked component of $\partial_* B_k \cup B_{k-1}$ we get
\[ |\partial_*(\partial_* B_k \cup B_{k-1}) \cap X| \geq |(\partial_* B_k \cup B_{k-1}) \cap X| + c(|\partial_*(\partial_* B_k \cup B_{k-1})| -|[\partial_* B_k \cup B_{k-1}]|) \]
\[ \geq |B_{k-1}| + c\frac{\log n}{2k} |\partial_* B_k| > |B_{k-1}| + \frac{1}{k} |\partial_* B_k| = |B_{k-1}| + |B_k| ,\]
as desired, where the last equality holds because $B_k$ consists entirely of 2-linked components of size $1$. Thus $k \leq \ceil{n/2}+1$.

A similar argument yields that the lowest layer of $B$ is in layer number at least $\floor{n/2}-1$. If $n$ is even, then since $B_{n/2+1}$ and $B_{n/2-1}$ consist entirely of nearly isolated sets, we are done.

Now assume $n$ is odd. By Theorem~\ref{thm::KK}, one of $[B_{\ceil{n/2}}]$ or $[B_{\floor{n/2}}]$ has size at most $\frac{1}{2} \binom{n}{\ceil{n/2}}$. Without loss of generality, suppose it is $B_{\ceil{n/2}}$ whose closure has size at most $\frac{1}{2} \binom{n}{\ceil{n/2}}$. Then the above argument works similarly for $k=\ceil{n/2}+1$, using Theorem~\ref{thm::2-layers} instead of Lemma~\ref{lem::2-layers}, where we obtain lower bounds on
\[ |\partial_*(\partial_* B_k \cup B_{k-1})| - |[\partial_* B_k \cup B_{k-1}]| \]
required for Theorem~\ref{thm::2-layers} from $|B_k| \leq 2^{n/2}$, $|[B_{\ceil{n/2}}]| \leq \frac{1}{2} \binom{n}{\ceil{n/2}}$, and Lemma~\ref{lem::iso}. Thus $B_{\ceil{n/2}+1}$ is empty. Similar to an argument at the beginning of the proof, since $|[B_{\ceil{n/2}}]| \leq \frac{1}{2} \binom{n}{\ceil{n/2}}$, we get that $B_{\ceil{n/2}}$ consists entirely of nearly isolated sets by Theorem~\ref{thm::2-layers}. This means that $B$ consists of $B_{\floor{n/2}}$ together with nearly isolated vertices in $\binom{[n]}{\floor{n/2}+1}$ and $\binom{[n]}{\floor{n/2}-1}$.
\end{proof}

\section{When \texorpdfstring{$a$}{a} is small}\label{sec::large-delta}

We make effective use of containers when $t = |N(A)|-|A|$ is large in Lemmas~\ref{lem::weak} and~\ref{lem::strong}. When $t$ is small, we must have that $a$ is small, and we have no need for the containers and can compute the cost of $A$ directly. We first recall a well-known lemma.

\begin{lemma}[\cite{HK2} Proposition 2.6]\label{lem::2-linked cost}
Let $G$ be a graph with maximum degree at most $d$ and let $v$ be a vertex of $G$. Then the number of trees in $G$ rooted at $v$ with $a$ vertices is at most $(ed)^{a-1}$.
\end{lemma}

We recall the definition of the graph $M$ as the bipartite graph with parts $L_k = \binom{[2k-1]}{k}$ and $L_{k-1} = \binom{[2k-1]}{k-1}$ and edge relation given by containment. We need a stronger version of Lemma~\ref{lem::iso} for small $a$, which follows from the fact that $M$ is $C_4$-free.

\begin{lemma}\label{lem::iso-small-a}
For $A \subseteq L_k$ with $|A| = a \leq k$, we have $|N(A)| \geq ka - \binom{a}{2}$.
\end{lemma}

Now we can prove Theorem~\ref{thm::2-layers} for small $A$.

\begin{lemma}\label{lem::small a}
Fix a constant $\ell \in \mathbb{Z}^+$. Let $p > 1-2^{-2/\ell}$ be constant, and let $X = V(M)_p$. Then there exists a constant $c>0$ such that with probability at least $1-o(1/k)$, every closed, 2-linked $A \subseteq L_k$ with $\ell \leq |A| \leq \log^8 k$ satisfies $|N(A) \cap X| - |A \cap X| \geq c(|N(A)|-|A|)$. 
\end{lemma}

\begin{proof}
Let $J$ be the graph with vertex set $L_k$ and adjacency given by the 2-linked relation: $u$ is adjacent to $v$ if and only if $N(u) \cap N(v) \neq \emptyset$. Note that $J$ is a $k(k-1)$-regular graph. We associate to each 2-linked $A \subseteq L_k$ of size $a$ an arbitrary rooted tree $T$ in $J$ with $V(T) = A$. By Lemma~\ref{lem::2-linked cost} and Proposition~\ref{prop::costfacts}, we have
\begin{equation}\label{eq::small a cost}
\cost(A) \leq \cost(T) \leq \log_2 |L_k| + (a-1)\log_2(ek(k-1)) \leq 2k + 3 a \log_2 k ,
\end{equation}
where the $\log_2 |L_k|$ term comes from the number of choices for the root of $T$.

Let $p = 1-2^{-2(1+\ep)/\ell} > 1-2^{-2/\ell}$, where $\ep>0$ is constant, and let $\delta = \min\{\frac{\ep/2}{1+\ep},(\frac{\ep}{6\ell})^2\}$. By Lemma~\ref{lem::iso-small-a}, the probability that $|N(A) \cap X| \leq \delta ka$ is at most
\[ \sum_{i\leq \delta ka} \binom{ka - a^2/2}{i} p^{i} (1-p)^{ka - \binom{a}{2} - i} \leq \exp_2\left( \delta k a \log_2(4/\delta) - \frac{2}{\ell}(1+\ep)(1-\delta)ka \right) \]
\[ \leq \exp_2\left( 3\sqrt{\delta} ka - \frac{2}{\ell}(1+\ep/2) ka \right) \leq \exp_2 \left( -\frac{2}{\ell}(1+\ep/4) ka \right) .\]
Using~\eqref{eq::small a cost} to take a union bound, we conclude that for all 2-linked $A$ with $|A| = a \geq \ell$ we have $|N(A)| > \delta ka$, with probability at least
\[ 1 - \exp_2\left( 2k + 3a \log_2 k - \frac{2}{\ell} (1+\ep/4)ka \right) \geq 1 - \exp_2\left( -\frac{\ep ka}{3\ell} \right) .\]
Taking a union bound over all $\ell \leq a \leq \log^8 k$, we conclude that
\[ |N(A) \cap X| - |A \cap X| \geq \delta ka - a \geq \frac{\delta}{2} ka \geq \frac{\delta}{2}(|N(A)|-|A|) \]
for all 2-linked $A$ with $\ell \leq a \leq \log^8 k$ with probability $1-o(1/k)$.
\end{proof}

\section{Weak containers}\label{sec::weak}

Recall that we work in the ``middle two layers'' graph $M$, where $A \subseteq L_k$ and $N(A) \subseteq L_{k-1}$. We let $G = N(A)$, $a = |A|$, and $t = |N(A)|-|A| \geq \max\{a/2k, k\log^7 k\}$.

\begin{proof}[Proof of Lemma~\ref{lem::weak}]
We define our weak containers in several steps, conveniently summarized in Figure~\ref{fig::weak}. We show the existence and other nice properties of the intermediary sets by claims throughout the proof. In between the claims, we bound the cost of the intermediary sets and show that they are concentrated.

Let $G^h$ be the set of vertices $v \in L_{k-1}$ for which there are at least $k^3/(2\log^4 k)$ walks $vxyz$ where $x,z \in A$. Note that $G^h \subseteq G$. For $R \subseteq A$, define $T = N(N(N(R)))$, $F' = N(N(N(R)) \cap A)$, and $L = E(N(R), \bar A)$,\footnote{The choice of notation is meant to be a helpful reminder of the definition: $R$ will be a \emph{random} subset of $A$, $T$ is the \emph{triple} neighborhood of $R$, and $L$ is the set of edges which \emph{leak} from $N(R)$ to $\bar A$. The letters $G$, $S$ and $F$ are historical convention (see e.g.~\cite{Galvin}).} where $\bar{A} = L_k \setminus A$ is the complement of $A$.

\begin{claim}
For every $A$, there exists a choice of $R$ such that\footnote{Note that $R$ may be empty, but the following arguments still go through.}
\begin{equation}\label{eq::R}
|R| \leq 18 a \frac{\log^5 k}{k^3} ,
\end{equation}
\begin{equation}\label{eq::L}
|L| \leq 18 t \frac{\log^5 k}{k} ,
\end{equation}
\begin{equation}\label{eq::Gh-F'}
|G^h \setminus F'| \leq 9\frac{t}{k^{1/5}} .
\end{equation} 
\end{claim}

\begin{proof}
Let $R = A_q$ with $q = \frac{6\log^5 k}{k^3}$, that is, we choose each element of $A$ to be included in $R$ independently with probability $q$. We first find upper bounds on the expected sizes of $R$, $L$, and $G^h \setminus F'$, and then by Markov's inequality, conclude the existence of an $R$ that yields such upper bounds within a constant factor.

Clearly $\mathbb{E}|R| = aq$. Recalling that $t = |G| - |A|$, we have
\begin{equation}\label{eq::egabar}
|E(G,\bar{A})| = tk .
\end{equation}
For each edge $e \in E(G,\bar{A})$, the probability that $e \in L$ is the probability that the endpoint of $e$ in $G$ has a neighbor which is in $R$, which is at most $kq$. Thus $\mathbb{E}|L| \leq tk \cdot kq \leq 6 t \log^5 k/k$. Finally we bound $\mathbb{E}|G^h \setminus F'|$. For $v \in G^h$, consider all the vertices $z$ which are endpoints of a walk $vxyz$ with $x,z \in A$. If the distance between $v$ and $z$ is $1$, then all such walks have $x=z$ or $v=y$, since every two vertices of $L_k$ share at most one neighbor in $L_{k-1}$, and hence there are at most $2k^2$ such walks over all choices of $z$. Since $v \in G^h$, there are at least $\frac{k^3}{2\log^4 k} - 2k^2$ \emph{paths} $vxyz$ with $x,z \in A$, that is, walks $vxyz$ with all vertices distinct. This gives at least $\frac{1}{2} \big(\frac{k^3}{2\log^4 k} - 2k^2 \big) \geq \frac{k^3}{5 \log^4 k}$ vertices $z \in A$ which are the endpoints of such paths, because between any two vertices at distance $3$ from each other in $M$ there are exactly two paths of length $3$. Thus
\[ \mathbb{E}|G^h \setminus F'| \leq |G| (1-q)^{k^3/5\log^4 k} \leq (a+t) \exp\left( - q \frac{k^3}{5\log^4 k} \right) = \frac{a+t}{k^{6/5}} \leq 3t/k^{1/5} ,\]
where in the last step we used the assumption that $t \geq a/2k$.

By Markov's inequality, there exists a choice of $R$ as desired.
\end{proof}

Using~\eqref{eq::R} and simply the fact that $R \subseteq L_k$, we compute the cost of $R$ as
\[ \cost(R) \leq \log_2 \binom{\binom{2k-1}{k}}{\leq O(a\log^5 k / k^3)} \leq O\left( \frac{a \log^5 k}{k^3} \log\frac{e\binom{2k-1}{k}}{O(a\log^5 k/k^3)} \right) \]
\begin{equation}\label{eq::costR}
\leq O\left( a \frac{\log^5 k}{k^3} \left( \log\frac{\binom{2k-1}{k}}{a} + 3\log k\right) \right) \leq O\left( t \frac{\log^6 k}{k^2} \right) ,
\end{equation}
where the last inequality follows by Lemma~\ref{lem::iso} and $t \geq a/2k$.

Also from~\eqref{eq::R}, we trivially have
\begin{equation}\label{eq::T}
|T| \leq |R| \cdot k^3 \leq O(a \log^5 k) .
\end{equation}
Since $T$ is determined by $R$, we have $\cost(T) \leq \cost(R)$, so we may conclude that $T$ is concentrated by Proposition~\ref{prop::Chern} after checking that $t^2/|T|$ is both $\omega(\cost(R))$ and $\omega(k)$. Using~\eqref{eq::costR}, \eqref{eq::T}, and $t \geq a/2k$, we have
\[ \frac{t^2}{|T|} \geq \Omega\left( \frac{t^2}{a \log^5 k} \right) \geq \Omega\left( \frac{t}{k \log^5 k} \right) \geq \omega\left( t \frac{\log^6 k}{k^2} \right) \geq \omega(\cost(R)) .\]
Notice also that if $t \geq k^3$, we have $t^2/|T| = \omega(k)$. Otherwise, we have $a \leq 2kt \leq k^4$, so by~\eqref{eq::T}, Lemma~\ref{lem::iso}, and our assumption that $t \geq k\log^7 k$ we have
\[ \frac{t^2}{|T|} \geq \Omega\left( \frac{t^2}{a \log^5 k} \right) \geq \Omega\left( \frac{t}{k \log^5 k} \log\frac{\binom{2k-1}{k}}{a} \right) \geq \Omega\left( \frac{t}{\log^5 k} \right) \geq \omega(k) .\]

Since $L$ is a set of edges incident to $N(R)$, we have by~\eqref{eq::R}, \eqref{eq::L}, and $t \geq a/2k$ that
\[ \cost(L|R) \leq \log_2 \binom{k\cdot |N(R)|}{\leq O\left(t\log^5k/k\right)} \leq O\left( t \frac{\log^5 k}{k} \log\frac{k^2 \cdot |R|}{t \log^5 k / k} \right) = O\left( t \frac{\log^6 k}{k} \right) .\]
Because $F'$ is determined by $R$ and $L$, we have by Proposition~\ref{prop::costfacts} that
\begin{equation}\label{eq::costF'}
\cost(F') \leq \cost(R) + \max_R \cost(L|R) = O\left( t \frac{\log^6 k}{k} \right) .
\end{equation}
By the definitions of $T$, $L$, and $F'$, and by~\eqref{eq::L}, we have
\begin{equation}\label{eq::T-F'}
|T \setminus F'| \leq |L| \cdot k = O(t\log^5 k) ,
\end{equation}
so $T \setminus F'$ is concentrated by Proposition~\ref{prop::Chern}, because
\[ \frac{t^2}{|T \setminus F'|} \geq \Omega\left( \frac{t}{\log^5 k} \right) \geq \max\left\{\omega\left( \cost(T,F') \right), \omega(k)\right\} ,\]
using our assumption that $t \geq k \log^7 k$.
By Proposition~\ref{prop::unionconc}, together with the fact that $T$ is concentrated, we have that $F'$ is concentrated.

For a vertex $v$ and a set of vertices $Y$, we define $d(v,Y) = |N(v) \cap Y|$ to be the \emph{degree of $v$ to $Y$}.
Let $Q = \{v : d(v,T) \geq k/2\}$, so
\begin{equation}\label{eq::Q}
|Q| \leq \frac{k|T|}{k/2} = 2|T| \leq O(a \log^5 k) 
\end{equation}
by~\eqref{eq::T}. Since $T$ determines $Q$, we have that $Q$ is concentrated by Proposition~\ref{prop::Chern}, in the same way that we showed that $T$ is concentrated. Let $D = \{v \in Q \setminus A : d(v,T\setminus G) \geq k/4\}$.

\begin{claim}
For every $A$, there exists a choice of $R' \subseteq T \setminus G$ such that
\begin{equation}\label{eq::R'}
|R'| \leq O(t \log^6 k / k) ,
\end{equation}
\begin{equation}\label{eq::D-N(R')}
|D \setminus N(R')| \leq O(t) .
\end{equation}
\end{claim}

\begin{proof}
Let $R' = (T \setminus G)_{q'}$, where $q' = 5 \frac{\log k}{k}$. By~\eqref{eq::T-F'} and the fact that $F' \subseteq G$, we have $\mathbb{E}|R'| \leq O(t \log^6 k / k)$. Then using $|D| \leq |Q| \leq O(a \log^5 k)$ from~\eqref{eq::Q}, we have
\[ \mathbb{E}|D \setminus N(R')| \leq |D|(1-q')^{k/4} \leq O(a\log^5 k) \cdot \exp\left( -q' \frac{k}{4} \right) = O\left( a \frac{\log^5 k}{k^{5/4}} \right) \leq O(t) .\]
By Markov's inequality, there exists a choice of $R'$ as desired.
\end{proof}

Let $Q' = Q \setminus N(R')$. Since $R' \subseteq T$, we have by~\eqref{eq::R'} and $t \geq a/2k$ that
\begin{equation}\label{eq::costR'}
\cost(R'|T) \leq \log_2 \binom{|T|}{\leq O(t \log^6 k / k)} \leq O\left( t \frac{\log^6 k}{k} \log\frac{a \log^5 k}{t \log^6 k / k} \right) \leq O\left( t \frac{\log^7 k}{k} \right) .
\end{equation}
Since $\cost(R',Q) = \cost(R'|T) + \cost(R)$ and $t \geq k\log^7 k$, we have by~\eqref{eq::R'},
\[ \frac{t^2}{|Q \cap N(R')|} \geq \frac{t^2}{|N(R')|} \geq \Omega\left( \frac{t}{\log^6 k} \right) \geq \max\left\{\omega(\cost(R',Q)), \omega(k)\right\} .\]
Thus $Q \cap N(R')$ is concentrated, and hence by Proposition~\ref{prop::unionconc}, $Q'$ is concentrated.

Finally, let $S' = Q' \cup A$. In the language of Lemma~\ref{lem::weak}, $\mathcal{W}_t'$ is the set of all $(S',F')$ pairs produced in the above fashion. By definition, the first two bullets of Lemma~\ref{lem::weak} are satisfied. We have also already shown the last bullet, namely that $F'$ is concentrated. It remains to be shown that the third and fourth bullets hold --- that is, that $|S'| \leq |F'| + O(t)$ and $S'$ is concentrated --- and that the total cost of $S'$ and $F'$ is $o(t)$.

Since $R' \cap G = \emptyset$, we have that $A \setminus Q' = A \setminus Q$. Thus to show that $S'$ is concentrated, it suffices to show that $|A \setminus Q| = o(t)$ by Propositions~\ref{prop::autoconc} and~\ref{prop::unionconc}. The following claim proves something slightly stronger that is needed for later and also proves that $(S',F')$ is a weak container.

\begin{figure}
\centering
\begin{tikzpicture}
\begin{scope}[shift={(0,0)}]
	\node[draw] (R) at (0,0) {$R = A_q$};
	\node (T) at (-3,-1) {$T = N(N(N(R)))$};
	\node[draw] (L) at (6,-2) {$L = E(N(R),\bar{A})$};
	\node (F') at (4,-4) {$F' = N(N(N(R)) \cap A)$};
	\node (Q) at (-6,-2.5) {$Q = \{v : d(v,T) \geq k/2\}$};
	\node[draw] (R') at (-1,-2.5) {$R' = (T \setminus G)_{q'}$};
	\node (Q') at (-3,-4) {$Q' = Q \setminus N(R')$};
	\node[draw] (S') at (-3,-5.5) {$S' = Q' \cup A$};

	\draw (R) -- (T);
	\draw (R) -- (L);
	\draw (L) -- (F');
	\draw (R) -- (F');
	\draw (T) -- (Q);
	\draw (T) -- (R');
	\draw (Q) -- (Q');
	\draw (R') -- (Q');
	\draw (Q') -- (S');
\end{scope}
\end{tikzpicture}
\caption{The intermediary sets used to construct weak containers. Each set can be defined by what is immediately above it in the diagram, except for the boxed sets, which also require $A$ or $G = N(A)$ for their definition. We compute the cost of an intermediary set $Y$ by computing $\cost(Y|Z)$, where $Z$ is all the sets above $Y$, and adding it to $\cost(Z)$. The cost of the boxed sets is kept low because the size of the set that depends on $A$ or $G$ is appropriately small. The cost of the other sets is at most the cost of the sets above them.}
\label{fig::weak}
\end{figure}
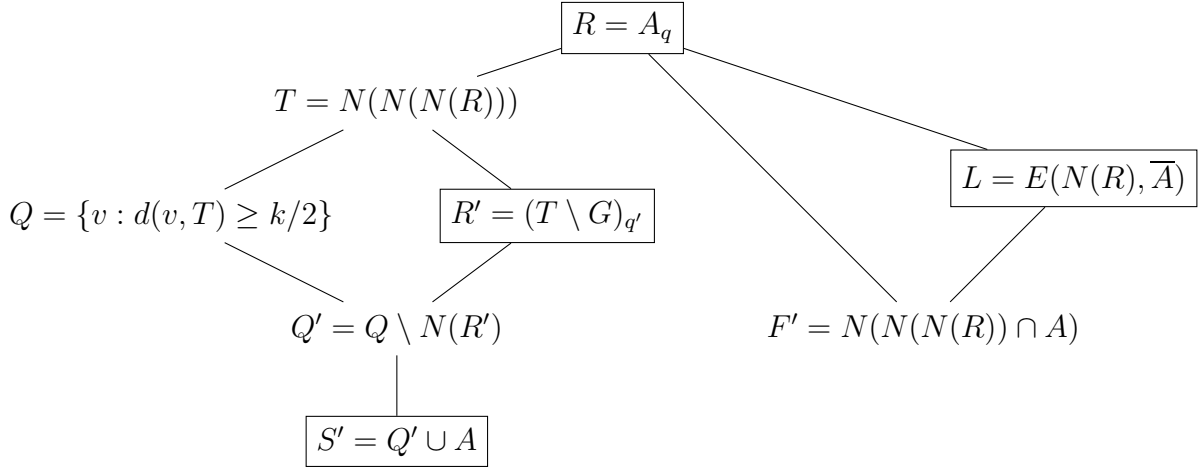

\begin{claim}\label{claim::A-Q}
For every $A$, we have that $|S' \setminus A| = O(t)$, $|G \setminus F'| = O(t)$, and $|A \setminus Q| = o(t/\log k)$.
\end{claim}

\begin{proof}
Each vertex of $(Q \setminus D) \setminus A$ sends at least $k/4$ edges to $G$, so by~\eqref{eq::egabar}, we have $|(Q \setminus D) \setminus A| = O(t)$. Combining with~\eqref{eq::D-N(R')}, we have
\[ |S' \setminus A| = |Q' \setminus A| = |(Q \setminus N(R')) \setminus A| \leq |(Q \setminus D) \setminus A| + |(D \setminus N(R')) \setminus A| = O(t) .\]

To show that $|G \setminus F'| = O(t)$, by~\eqref{eq::Gh-F'}, it suffices to show that $|G \setminus G^h| = O(t)$. Define the following sets:
\[ G^\ell = \{ v \in G : d(v,A) < k/\log^2 k \}, \quad \quad A^\ell = \{ v \in A : d(v,G^\ell) > k/2 \} ,\]
\[ G^{\ell\ell} = \{ v \in G : d(v,A\setminus A^\ell) < k/\log^2 k \}, \quad \quad A^{\ell\ell} = \{ v \in A : d(v,G^{\ell\ell}) \geq k/4 \} .\]
Each vertex of $G^\ell$ sends at least $(1-o(1))k$ edges to $\bar{A}$, so by~\eqref{eq::egabar},
\[ |G^\ell| \leq \frac{tk}{(1-o(1))k} = O(t) .\]
Similarly, we have
\[ |A^\ell| \leq \frac{|G^\ell|k/\log^2 k}{k/2} = O\left(\frac{t}{\log^2 k}\right) ,\]
\[ |G^{\ell\ell}| \leq \frac{tk+|A^\ell|k}{(1-o(1))k} = O(t) ,\]
\begin{equation}\label{eq::All}
|A^{\ell\ell}| \leq |A^\ell| + \frac{|G^{\ell\ell}|k/\log^2 k}{k/4} = O\left(\frac{t}{\log^2 k}\right) .
\end{equation}
If $v \in G \setminus G^{\ell\ell}$, then the number of $vxyz$ walks with $x,z \in A$ is at least $\frac{k}{\log^2 k} \cdot \frac{k}{2} \cdot \frac{k}{\log^2 k} = \frac{k^3}{2\log^4 k}$, so $v \in G^h$. This means $G \setminus G^h \subseteq G^{\ell\ell}$, giving $|G \setminus G^h| = O(t)$ as desired.

Note that if $v \in A \setminus Q$ has at least $k/4$ neighbors in $G \setminus G^h \subseteq G^{\ell\ell}$, then $v \in A^{\ell\ell}$. Otherwise, by the definition of $Q$, $v$ has at least $k/4$ neighbors in $G^h \setminus T \subseteq G^h \setminus F'$. Thus by~\eqref{eq::Gh-F'} and~\eqref{eq::All}, we have
\[ |A \setminus Q| \leq |A^{\ell\ell}| + \frac{|G^h \setminus F'|k}{k/4} = o(t/\log k) . \qedhere\]
\end{proof}

With Claim~\ref{claim::A-Q}, we have thus far established that the $(S',F')$ pairs are concentrated weak containers. Recall that the cost of $F'$ is $o(t)$ by~\eqref{eq::costF'}. Using Figure~\ref{fig::weak} as a guide and recalling the cost of $R$ from~\eqref{eq::costR} and the cost of $R'$ from~\eqref{eq::costR'}, we have
\[ \cost(S') \leq \cost(A \setminus Q' | Q') + \cost(R' | R) + \cost(R) \leq \cost(A \setminus Q' | Q') + o(t) .\]
All that is left to do is show that $\cost(A \setminus Q' | Q') = o(t)$. This is the only time we use that $A$ is 2-linked. Let $J$ be the graph with vertex set $L_k$ and adjacency given by the 2-linked relation: $u$ is adjacent to $v$ if and only if $N(u) \cap N(v) \neq \emptyset$. Note that $J$ is $k(k-1)$-regular, and $A$ induces a connected subgraph of $J$. Assuming $A \cap Q' \neq \emptyset$, we associate to $A$ a rooted forest in $J$, rooted in $A \cap Q'$, whose non-roots are exactly the vertices of $A \setminus Q'$. The number of choices for $A \setminus Q'$, given $Q'$, is at most the number of choices for such forests, given $Q'$. Assuming there are $r$ roots, there are at most $\binom{|Q'|}{r}$ choices for roots, at most $\binom{|A \setminus Q'|}{r}$ ways to assign sizes of the trees in the forest to the roots, and at most $(ek(k-1))^{|A \setminus Q'|}$ ways to grow the trees by Lemma~\ref{lem::2-linked cost}. By Claim~\ref{claim::A-Q}, we have $r \leq |A \setminus Q'| = o(t/\log k)$. Given that there are $r$ roots, we have
\[ \cost(A \setminus Q' | Q', r) \leq \log_2\binom{|Q'|}{r} + \log_2 \binom{|A \setminus Q'|}{r} + |A\setminus Q'| \log_2(ek(k-1)) \]
\[ \leq r \log_2\frac{e|Q'|}{r} + |A \setminus Q'| + 3|A \setminus Q'| \log_2 k \leq |A \setminus Q'| \log_2\frac{e|Q'|}{|A \setminus Q'|} + 4|A \setminus Q'| \log_2 k .\]
The last term is always $o(t)$ since $|A \setminus Q'| = o(t/\log k)$. The second-to-last term is clearly $o(t)$ if $\frac{e|Q'|}{|A \setminus Q'|} \leq k^3$, say, while otherwise we have that $|A \setminus Q'| \leq \frac{e|Q'|}{k^3} = O(a\log^5 k / k^3)$ by~\eqref{eq::Q} and $Q' \subseteq Q$, and hence using $t \geq a/2k$ we get
\[ |A \setminus Q'| \log_2\frac{e|Q'|}{|A \setminus Q'|} \leq O\left( a \frac{\log^5 k}{k^3} \right) \cdot \log_2 (e|L_k|) = O\left( t \frac{\log^5 k}{k} \right) = o(t) .\]
Adding in the $\log|A \setminus Q'|$ cost of choosing $r$, the number of roots, we have that $\cost(A \setminus Q' | Q') = o(t)$ if $A \cap Q' \neq \emptyset$.

If $A \cap Q' = \emptyset$, then we associate to each $A$ an arbitrary rooted spanning tree in $J$ whose vertex set is $A$. The arbitrary root comes with an additional cost of $\log_2 |L_k| = O(k) = o(t)$ over the case when $A \cap Q' \neq \emptyset$, since $t \geq k \log^7 k$.
\end{proof}

\section{Strong Containers}\label{sec::strong}

Recall that we work in the ``middle two layers'' graph $M$, where $A \subseteq L_k$ and $N(A) \subseteq L_{k-1}$. Recall that $G = N(A)$, $a = |A|$, and $t = |N(A)|-|A| \geq \max\left\{a/2k, k\log^7 k\right\}$.

\begin{proof}[Proof of Lemma~\ref{lem::strong}]
Suppose we have concentrated weak containers $(S', F')$ at cost $o(t)$, as guaranteed by Lemma~\ref{lem::weak}. We transform these $(S',F')$ into concentrated strong containers $(S,F)$ at an additional cost of $o(t)$. See Figure~\ref{fig::strong} for a convenient description of the dependencies of the intermediary sets. Let $\psi = \log^2 k$.

Recall that, for a vertex $v$ and a set of vertices $Y$, we define $d(v,Y) = |N(v) \cap Y|$ to be the \emph{degree of $v$ to $Y$}. Let $H$ be a minimum set of vertices of $A$ such that $F'' := F' \cup N(H)$ satisfies that for every $v \in A$, $d(v,G \setminus F'') \leq \psi$.

\begin{claim}\label{claim::H}
For every $A$, we have $|H| \leq O(t/\psi)$.
\end{claim}

\begin{proof}
It suffices to construct some such $H$ which has size $O(t/\psi)$. We construct $H$ greedily, considering vertices $v \in A$ one at a time and placing them in $H$ if $d(v,G \setminus (F' \cup N(H))) > \psi$. Upon the addition of $v$ to $H$, we reduce the size of $G \setminus (F' \cup N(H))$ by at least $\psi$. Thus $|H| \leq \frac{|G \setminus F'|}{\psi}$, which is $O(t/\psi)$ since $(S',F')$ is a weak container.
\end{proof}

As $H \subseteq S'$, we have by Claim~\ref{claim::H} and $t \geq a/2k$ that
\[ \cost(H|S') \leq \log_2\binom{|S'|}{\leq O(t/\psi)} \leq O\left( \frac{t}{\psi} \log\frac{tk}{t/\psi} \right) \leq O\left(t \frac{\log k}{\psi}\right) \leq o(t) .\]
Since $F'' \subseteq G$, we have $|N(H) \setminus F'| = O(t)$, and hence
\[ \frac{t^2}{|N(H) \setminus F'|} \geq \Omega(t) \geq \max\left\{\omega(\cost(H|S') + \cost(S',F')), \omega(k)\right\} ,\]
recalling that $t \geq k \log^7 k$. Thus $N(H) \setminus F'$ is concentrated by Proposition~\ref{prop::Chern}. Since $F'$ is concentrated, by Proposition~\ref{prop::unionconc}, we have that $F''$ is concentrated.

Let $S'' = \{v : d(v,F'') \geq k-\psi\}$. From the choice of $H$, we have that $A \subseteq S''$. Furthermore, by~\eqref{eq::egabar},
\begin{equation}\label{eq::S''-A}
|S'' \setminus A| \leq \frac{tk}{k-\psi} = O(t) .
\end{equation}
Since $|S' \setminus A| = O(t)$, we have $|S' \triangle S''| = O(t)$. The costs of $S'$ and $S''$ are both $o(t)$, so by Proposition~\ref{prop::Chern}, both $S' \setminus S''$ and $S'' \setminus S'$ are concentrated. Thus $S''$ is concentrated by Proposition~\ref{prop::unionconc}.

Now let $U$ be a minimum set of vertices of $\bar{G}$ such that $S := S'' \setminus N(U)$ satisfies that for every $v \in \bar{G}$, $d(v,S) \leq \psi$, where $\bar{G} = L_{k-1} \setminus G$ is the complement of $G$.

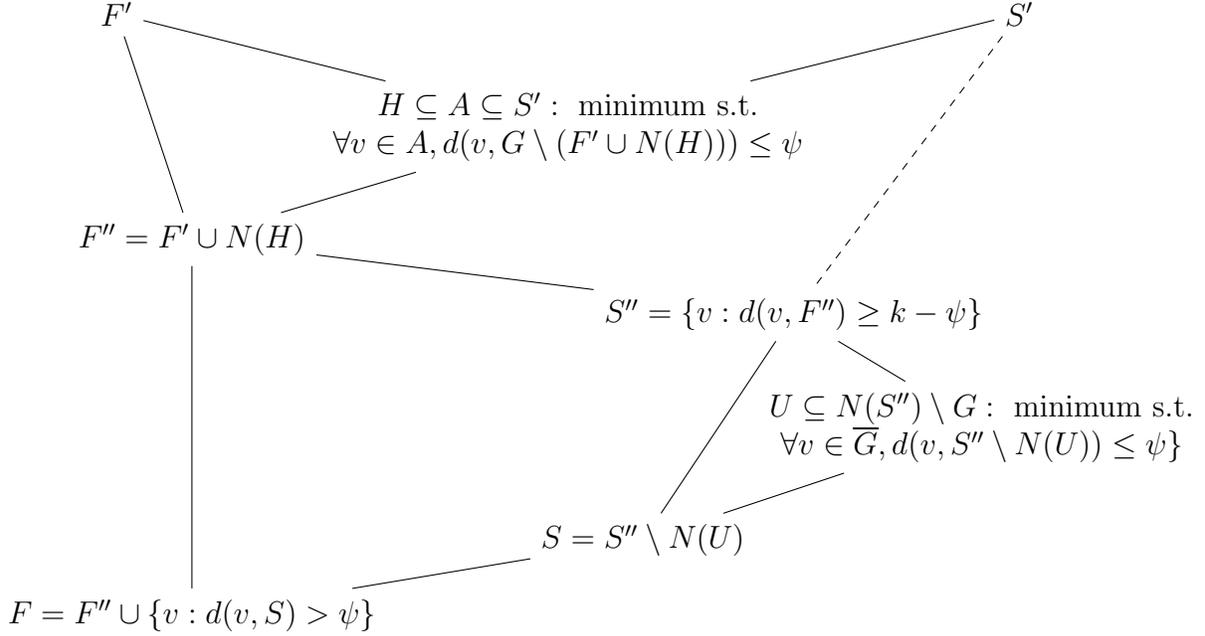
\begin{figure}[ht]
\centering
\begin{tikzpicture}
\begin{scope}[shift={(0,0)}]
	\node (F') at (-4,0) {$F'$};
	\node (S') at (8,0) {$S'$};
	\node[align=center] (H) at (2,-1.5) {$H \subseteq A \subseteq S' : \text{ minimum s.t.}$ \\ $\forall v \in A, d(v,G\setminus (F' \cup N(H))) \leq \psi$};
	\node (F'') at (-3,-3) {$F'' = F' \cup N(H)$};
	\node (S'') at (5,-4) {$S'' = \{v : d(v,F'') \geq k-\psi\}$};
	\node[align=center] (U) at (7.5,-5.5) {$U \subseteq N(S'') \setminus G : \text{ minimum s.t.}$ \\ $\forall v \in \bar{G}, d(v,S''\setminus N(U)) \leq \psi\}$};
	\node (S) at (3,-7) {$S = S'' \setminus N(U)$};
	\node (F) at (-3,-8) {$F = F'' \cup \{v : d(v,S) > \psi\}$};
	
	\draw (F') -- (H);
	\draw (S') -- (H);
	\draw (F') -- (F'');
	\draw (H) -- (F'');
	\draw (F'') -- (S'');
	\draw[dashed] (S') -- (S'');
	\draw (S'') -- (U);
	\draw (S'') -- (S);
	\draw (U) -- (S);
	\draw (S) -- (F);
	\draw (F'') -- (F);
\end{scope}
\end{tikzpicture}
\caption{The intermediary sets used to construct strong containers from weak containers. Each set can be defined by what is immediately above it via solid lines in the diagram. The dashed line represents a useful dependency when verifying that $S''$ is concentrated.}
\label{fig::strong}
\end{figure}

\begin{claim}\label{claim::U}
For every $A$, we have $|U| \leq O(t/\psi)$ and $U \subseteq N(S'')$.
\end{claim}

\begin{proof}
The proof is similar to the proof of Claim~\ref{claim::H}. We may construct such $U$ greedily, by considering vertices $v \in \bar{G}$ one at a time and placing them in $U$ if $d(v,S'' \setminus N(U)) > \psi$. Upon the addition of $v$ to $U$, we reduce the size of $S'' \setminus N(U)$ by at least $\psi$. Since $U \subseteq \bar{G}$, we have that $S'' \setminus N(U)$ always contains $A$, and hence $|U| \leq \frac{|S'' \setminus A|}{\psi} = O(t/\psi)$ by~\eqref{eq::S''-A}. That $U \subseteq N(S'')$ follows from the minimality of $U$: for each vertex $v \in U$, $N(v)$ should intersect $S''$.
\end{proof}

We have by Claim~\ref{claim::U} and $t \geq a/2k$ that
\[ \cost(U|S'') \leq \log_2\binom{|N(S'')|}{\leq O(t/\psi)} \leq O\left( \frac{t}{\psi} \log\frac{tk^2}{t/\psi} \right) \leq O\left( t \frac{\log k}{\psi} \right) \leq o(t) .\]
Since $U$ is disjoint from $G$, we have $A \subseteq S$ and $S'' \cap N(U) \subseteq S'' \setminus A$. By~\eqref{eq::S''-A}, we have
\[ \frac{t^2}{|S'' \cap N(U)|} \geq \Omega(t) \geq \max\left\{\omega(\cost(U|S'') + \cost(S'')), \omega(k)\right\} ,\]
so $S'' \cap N(U)$ is concentrated by Proposition~\ref{prop::Chern}. Hence by Proposition~\ref{prop::unionconc}, $S$ is concentrated.

Finally, define $F = F'' \cup \{v : d(v,S)>\psi\}$, so $\cost(S,F) = o(t)$. By the construction of $S$, we have $F \subseteq G$, and hence $|F \setminus F''| = O(t)$. By Proposition~\ref{prop::Chern}, $|F \setminus F''|$ is concentrated, so by Proposition~\ref{prop::unionconc}, $F$ is concentrated.

All that remains to be checked is that $|S| \leq |F| + o(t)$. By construction of $F$, every $v \in \bar{F}$ satisfies $d(v,\bar{S}) \geq k-\psi$. For every $v \in S$, we have $v \in S''$, so $d(v,F) \geq d(v,F'') \geq k-\psi$. Recall that $|E(G,\bar{A})| = |G|k - |A|k = tk$. Each vertex in $G \setminus F$ contributes at least $k-\psi$ edges to $E(G,\bar{A})$, so $|G \setminus F| \leq \frac{tk}{k-\psi}$, and each vertex in $S \setminus A$ contributes at least $k-\psi$ edges to $E(\bar{A},G)$, so $|S \setminus A| \leq \frac{tk}{k-\psi}$. Then we have
\[ k|F| + \frac{tk}{k-\psi} \psi \geq k|F| + \psi |G \setminus F| \geq e(S,G) \]
\[ \geq k|A| + (k-\psi)|S\setminus A| = k|S| - \psi |S\setminus A| \geq k|S| - \frac{tk}{k-\psi}\psi .\]
Rearranging, we obtain
\[ |S| \leq |F| + \frac{2t\psi}{k-\psi} \leq |F| + o(t) .\qedhere\]
\end{proof}

\paragraph{Remark.} It is in the final step of the proof of Lemma~\ref{lem::strong} that we use crucially that $M$ is regular. The proof would still go through if the bipartite graph were bi-regular, meaning the degrees on the same side are all equal, and the degrees of vertices in $A$ were larger than the degrees of vertices in $N(A)$.

\section{Using strong containers}\label{sec::use-containers}

\begin{proof}[Proof of Theorem~\ref{thm::2-layers}]
Let $p = 1/2+\ep$, where $\ep>0$ is a constant, and let $X = V(M)_p$. Lemma~\ref{lem::small a} proves Theorem~\ref{thm::2-layers} for $A$ with $a := |A| \leq \log^8 k$, so we assume that $a \geq \log^8 k$. For $a \leq k^3$, we use Theorem~\ref{thm::KK} with $a = \binom{z}{k}$ and $z = k+\alpha$ for some $\alpha \in (0,3]$ to see that
\[ t := |N(A)| - |A| \geq \binom{z}{k-1} - \binom{z}{k} = \frac{-\alpha+k-1}{\alpha+1} a \geq ak/4 \geq k \log^7 k .\]
For $a \geq k^3$, the assumption that $t \geq a/2k$ yields the same conclusion. Now fix $t \geq k \log^7 k$.

We now abandon the asymptotic notation to be careful with the constants. We have by Lemma~\ref{lem::strong} that concentrated strong containers exist: there exists $\mathcal{W}_t$ of size at most $2^{\ep t}$ such that with probability at least $1-\exp(-2k/\ep)$, for all $A \in \A_t$ there exists $(S,F) \in \mathcal{W}_t$ such that $A \subseteq S$, $F \subseteq N(A)$, $|S| \leq |F| + \frac{\ep^2}{16} t$, $||S \cap X|-p|S|| \leq \frac{\ep^2}{16} t$ and $||F \cap X|-p|F|| \leq \frac{\ep^2}{16} t$. Consider a given $A \in \A_t$ and the corresponding $(S,F) \in \mathcal{W}_t$, and suppose $X$ is such that this concentration of $S$ and $F$ hold.

If $|S| \leq |F| - \frac{\ep^2}{2} t$, then we are already done, since (recalling that $G = N(A)$)
\begin{equation}\label{eq::ineqchain}
|G \cap X| - |A \cap X| \geq |F \cap X| - |S \cap X| \geq p|F| - p|S| - \frac{\ep^2}{8} t \geq \frac{\ep^2}{8} t .
\end{equation}
Hence we may suppose $|S| \geq |F| - \frac{\ep^2}{2} t$. Our goal is to show that
\begin{equation}\label{eq::bad}
|(S \setminus A) \cap X| + |(G \setminus F) \cap X| \geq \frac{\ep^2}{4} t
\end{equation}
with probability at least $1-\exp(-2k/\ep)$ for all $A \in \A_t$, which improves~\eqref{eq::ineqchain} to
\[ |G \cap X| - |A \cap X| \geq \frac{\ep^2}{4}t + p(|F|-|S|) - \frac{\ep^2}{8} t \geq \frac{\ep^2}{16} t ,\]
using that $|S| \leq |F| + \frac{\ep^2}{16} t$.

First we determine the cost of $A$. The following idea is made explicit by Park~\cite{Par}, but appears in an earlier work of Hamm and Kahn~\cite{HK2}. For every strong container $(S,F) \in \mathcal{W}_t$ and every $a$ and $g$, fix $A^* \in \A_t$ of size $a$ and $G^* = N(A^*)$ of size $g$ which have $(S,F)$ as a strong container. Note that $A^*$ and $G^*$ are determined by $(S,F)$.

Given $A^*$, to determine $A$ it suffices to determine $A \setminus A^*$ and $A^* \setminus A$. The cost of $A \setminus A^*$ is at most $|S \setminus A^*|$, since $A \setminus A^* \subseteq S \setminus A^*$. Because $A$ is closed, $A^* \setminus A$ consists of all the neighbors of $G^* \setminus G$ in $A^*$, hence $A^* \setminus A$ is determined by $G^* \setminus G$. Thus the cost of $A^* \setminus A$ is at most $|G^* \setminus F|$, since $G^* \setminus G \subseteq G^* \setminus F$. Let $t' = |S \setminus A| + |G \setminus F| = t + (|S|-|F|)$. Then the cost of $A$ is at most $t' + \cost(S,F) \leq t' + \ep t$.

Now we bound the probability that~\eqref{eq::bad} occurs. Note that $|((S \setminus A) \cup (G \setminus F)) \cap X|$ is binomially distributed, so the probability that it is less than $\frac{\ep^2}{4} t$ is
\[ \sum_{i=0}^{\floor{\ep^2 t/4}} \binom{t'}{i} p^i (1-p)^{t'-i} = 2^{-t'} \sum_{i=0}^{\floor{\ep^2 t/4}} \binom{t'}{i} (1+2\ep)^i (1-2\ep)^{t'-i} \]
\[ \leq 2^{-t'} \frac{\ep^2 t}{4} \exp\left( \frac{\ep^2 t}{4} \log\frac{4et'}{\ep^2 t} + \frac{\ep^2 t}{4} \log(1+2\ep) + \left(t'-\frac{\ep^2 t}{4}\right)\log(1-2\ep) \right) \]
\[ \leq 2^{-t'} \frac{\ep^2 t}{4} \exp\left( \frac{\ep^2 t}{2} + \frac{\ep t}{2} + \ep^3 t - 2\ep t' \right) \leq 2^{-t'} \frac{\ep^2 t}{4} e^{-\ep t} .\]
Comparing to $\cost(A) \leq t' + \ep t$, we may take a union bound over all $A$ and choices of $a=|A|$ and $g=|G|$ to get that $|(S \setminus A) \cap X| + |(G \setminus F) \cap X| \geq \frac{\ep^2}{4} t$ for all $A \in \A_t$ with probability at least $1-\exp(\Omega(t)) \geq 1 - \exp(-2k/\ep)$. Finally, we take a union bound over all $k \log^7 k \leq t \leq \binom{2k-1}{k}$ to finish the proof.
\end{proof}

\end{document}